\documentclass[12pt,letterpaper,twoside]{article}

\usepackage{amsmath,amssymb}
\usepackage{amsthm}
\usepackage{stmaryrd}
\usepackage{fancyhdr}
\usepackage{times}
\usepackage{mathrsfs} 
\usepackage{titlesec}
\usepackage[pdftex,dvips]{geometry}

\newtheoremstyle{fact}
     {\topsep}
     {\topsep}
     {\slshape}
     {}
     {\bfseries}
     {}
     { }
     {\thmname{#1}\thmnumber{ #2.}\thmnote{ \rm (#3)}}

\newtheorem{theorem}{Theorem}[section]
\newtheorem{Ltheorem}{Theorem}

\newtheorem*{theorem*}{Theorem} 
\newtheorem{lemma}[theorem]{Lemma}
\newtheorem{proposition}[theorem]{Proposition}
\newtheorem{corollary}[theorem]{Corollary}
\newtheorem{problem}{Problem}

\theoremstyle{definition}
\newtheorem{definition}[theorem]{Definition}

\newtheorem*{remark*}{Remark}

\newtheorem*{question*}{Question}
\newtheorem*{examples*}{Examples}  
\newtheorem{example}[theorem]{Example}
\newtheorem*{example*}{Example}

\newtheorem*{convention*}{Convention}

\theoremstyle{fact}

\newtheorem{ftheorem}[theorem]{Theorem}

\newtheorem{fproposition}[theorem]{Proposition}

\newenvironment{myromanlist}[1][enumi]{\begin{list}{{\rm (\roman{#1})}}
{\usecounter{#1}\setlength{\labelwidth}{25pt}\setlength{\topsep}{-6pt}
\setlength{\itemsep}{-4pt} \setlength{\leftmargin}{25pt}}}{\end{list}}

\newenvironment{myalphlist}[1][enumi]{\begin{list}{{\rm (\alph{#1})}}
{\usecounter{#1}\setlength{\labelwidth}{25pt}\setlength{\topsep}{-6pt}
\setlength{\itemsep}{-4pt} \setlength{\leftmargin}{25pt}}}{\end{list}}

\def\proofont{\fontseries{bx}\fontshape{sc}\selectfont}
\def\proofname{Proof.}

\newcommand{\Note}[1]{}

\makeatletter
\renewenvironment{proof}[1][\proofname]{\par
  \normalfont
  \topsep6\p@\@plus6\p@ \trivlist
  \item[\hskip\labelsep\noindent\proofont #1]\ignorespaces
}{%
  \qed\endtrivlist
}
\makeatother

\titlelabel{\thetitle.\ }
\titleformat*{\section}{\normalsize\bfseries\centering}
\titleformat*{\subsection}{\normalsize\bfseries}
\titlespacing{\subsection}{0pt}{\topsep}{0.5ex}
\titleformat{\subsection}[runin]{\normalfont\bfseries}{%
\thesubsection.}{0.5ex}{}[.]

\author{D. Dikranjan\thanks{The first author acknowledges 
the financial aid received from MCYT, MTM2006-02036 and FEDER funds.} 
{ }and G\'abor Luk\'acs\thanks{The second author gratefully acknowledges 
the generous financial  support received from NSERC and the University of 
Manitoba, which enabled him to do this research.}}
   
\title{Locally compact abelian groups admitting\\ non-trivial quasi-convex  
null sequences\thanks{{\em 2000 Mathematics Subject Classification}: 
Primary 22B05, 22C05; Secondary 22A05; 54H11.\endgraf \hspace{5.5pt} 
{\em Keywords:} quasi-convex sets, quasi-convex null sequences, exotic 
torus, $p$-adic integers, torus group.}}

\begin{document}

\makeatletter
\def\@fnsymbol#1{\ifcase#1\or * \or 1 \or 2  \else\@ctrerr\fi\relax}

\let\mytitle\@title
\chead{\small\itshape D. Dikranjan and G. Luk\'acs / 
LCA groups admitting quasi-convex null sequences}
\fancyhead[RO,LE]{\small \thepage}
\makeatother

\maketitle

\def\thanks#1{} 

\thispagestyle{empty}


\begin{abstract}
\noindent
In this paper, we show that for every locally compact abelian group $G$, 
the following statements are equivalent:

\begin{myromanlist}

\item
$G$ contains no sequence $\{x_n\}_{n=0}^\infty$ such that
\mbox{$\{0\}\cup \{\pm x_n \mid n \in \mathbb{N} \}$}
is infinite and quasi-convex in $G$, and 
\mbox{$x_n \longrightarrow 0$};

\item
one of the subgroups $\{g \in G \mid 2g=0\}$ and 
$\{g \in G \mid 3g=0\}$ is open in $G$;

\item
$G$ contains an open  compact subgroup of the form $\mathbb{Z}_2^\kappa$ 
or $\mathbb{Z}_3^\kappa$ \ for some cardinal $\kappa$.

\end{myromanlist}
\end{abstract}

\section{Introduction}

\label{sect:intro}

One of the main sources of inspiration for the theory of topological 
groups is the theory of topological vector spaces, where the notion of 
convexity plays a prominent role. In this context, the reals $\mathbb{R}$ 
are replaced with the circle group $\mathbb{T}=\mathbb{R}/\mathbb{Z}$, and 
linear functionals are replaced by {\em characters}, that is, continuous 
homomorphisms to $\mathbb{T}$. By making substantial use of characters, 
Vilenkin introduced the notion of quasi-convexity for abelian topological 
groups as a counterpart of convexity in topological vector spaces (cf. 
\cite{Vilenkin}).  The counterpart of locally convex spaces are the 
locally quasi-convex groups.  This class includes all locally compact 
abelian groups and locally convex topological vector spaces 
(cf.~\cite{Banasz}).

According to the celebrated Mackey-Arens theorem 
(cf.~\cite{MackeyCTLS} and~\cite{ArensDLin}),
every locally convex topological vector space $(V,\tau)$ admits a 
so-called {\em Mackey topology}, that is, a locally convex vector 
space topology $\tau_\mu$ that is finest with respect to the 
property of having the same set of continuous linear functionals (i.e., 
\mbox{$(V,\tau_\mu)^*=(V,\tau)^*$}). Moreover, $\tau_\mu$ can be described 
as the 
topology of uniform convergence of the sets of an appropriate family of 
convex weakly compact sets of $(V,\tau)^*$. A counterpart of this notion 
in the class of locally quasi-convex abelian groups, the so-called {\em 
Mackey group topology},
was proposed in \cite{ChaMarTar}. It seems  reasonable to expect to  
describe the Mackey topology of a  locally quasi-convex abelian group $G$ 
as the topology $\tau_\mathfrak{S}$ of uniform convergence on members
of an appropriate family  $\mathfrak{S}$ of quasi-convex compact sets of 
the Pontryagin dual~$\widehat G$, where each set in $\mathfrak{S}$ is 
equipped with the weak topology. This underscores the importance 
of the compact quasi-convex sets in locally quasi-convex abelian groups. 
It was proven by Hern\'andez, and independently, by 
Bruguera and Mart\'{\i}n-Peinador, that a metrizable locally quasi-convex 
abelian group $G$ is complete if and only if the quasi-convex hull of 
every compact subset of $G$ is compact (cf.~\cite{Hern2} 
and~\cite{BrugMar2}). Luk\'acs extended this result, and proved that 
a~metrizable abelian group $A$ is MAP and has the quasi-convex compactness 
property if and only if it is locally quasi-convex and complete; he also 
showed that such groups are characterized by the property that the 
evaluation map $\alpha_A\colon A \rightarrow \hat{\hat A}$ is a closed 
embedding (cf.~\cite[I.34]{GLdualtheo}).

Let \mbox{$\pi\colon \mathbb{R} \rightarrow \mathbb{T}$}  denote  the 
canonical projection. Since the restriction 
\mbox{$\pi_{|[0,1)}\colon [0,1) \rightarrow \mathbb{T}$} is a~bijection, 
we often identify in the sequel, {\em par abus de language},  a~number 
\mbox{$a\in [0,1)$} with its image (coset) 
\mbox{$\pi(a)=a+\mathbb{Z}\in \mathbb{T}$}.  We put 
\mbox{$\mathbb{T}_m:=\pi([-\frac{1}{4m},\frac{1}{4m}])$} 
for all \mbox{$m\in \mathbb{N}\backslash\{0\}$}. According to standard 
notation in this area, we use $\mathbb{T}_+$ to denote $\mathbb{T}_1$.
For an abelian topological group $G$, we denote by $\widehat{G}$ 
the {\em Pontryagin dual} of a $G$, that is, the group of all 
characters of $G$ endowed with the compact-open topology.

\begin{definition}\label{def:into:qc}
For $E\subseteq G$ and $A \subseteq \widehat{G}$,  the 
{\em polars} of $E$ and $A$ are defined as
\begin{align}
E^\triangleright=\{\chi\in \widehat{G} \mid \chi(E) \subseteq 
\mathbb{T}_+\}
\quad \text{and} \quad
A^\triangleleft=\{ x \in A \mid  \forall \chi \in A, 
\chi(x) \in \mathbb{T}_+ \}.
\end{align}
The set $E$ is said to be {\em quasi-convex} if 
$E=E^{\triangleright\triangleleft}$. We say that $E$ is {\em $qc$-dense} 
if $G=E^{\triangleright\triangleleft}$.
\end{definition}

Obviously, \mbox{$E\subseteq E^{\triangleright\triangleleft}$} holds for 
every  \mbox{$E\subseteq G$}. Thus, $E$ is quasi-convex
if and only if for every \mbox{$x\in G\backslash E$} there exists 
\mbox{$\chi\in E^\triangleright$} such that 
\mbox{$\chi(x)\not\in \mathbb{T}_+$}. 
The set  \mbox{$Q_{G}(E):=E^{\triangleright\triangleleft}$} is the 
smallest quasi-convex set of $G$ that contains $E$, and it is 
called the {\em quasi-convex hull} of $E$. 

\begin{definition}
A sequence $\{x_n\}_{n=0}^\infty \subseteq G$ is said to be {\em  
quasi-convex} if  \mbox{$S=\{0\} \cup \{\pm x_n \mid n \in \mathbb{N}\}$} 
is quasi-convex in $G$. We say that $\{x_n\}_{n=0}^\infty$ is 
{\em non-trivial} if the set $S$ is infinite, and it is a {\em null 
sequence} if $x_n \longrightarrow 0$.
\end{definition}

\begin{example} \label{ex:intro:TJ23R}
Each of the compact groups $\mathbb{T}$, $\mathbb{J}_2$ ($2$-adic 
integers), and $\mathbb{J}_3$ ($3$-adic integers), and the locally compact 
group $\mathbb{R}$ admits a non-trivial quasi-convex null sequence
(cf.~\cite[1.2-1.4]{DikLeo} and~\cite[A-D]{DikGL1}).
\end{example}

In this paper, we characterize the locally compact abelian groups $G$ that 
admit a non-trivial quasi-convex null sequence.

\begin{samepage}
\begin{Ltheorem} \label{thm:qcs:main}
For every locally compact abelian group $G$, the following statements are 
equivalent:

\begin{myromanlist}

\item
$G$ admits no non-trivial quasi-convex null sequences;

\item
one of the subgroups $G[2]=\{g \in G \mid 2g=0\}$ and 
$G[3]=\{g \in G \mid 3g=0\}$ is open in $G$;

\item
$G$ contains an open compact subgroup of the form $\mathbb{Z}_2^\kappa$ or 
$\mathbb{Z}_3^\kappa$ \ for some cardinal $\kappa$.

{\parindent -25pt Furthermore, if $G$ is compact, then these conditions 
are also equivalent to:
}

\item
$G \cong \mathbb{Z}_2^\kappa \times F$ or 
$G\cong \mathbb{Z}_3^\kappa \times F$,
where $\kappa$ is some cardinal and $F$ is a finite abelian group;

\item
one of the subgroups $2G$ and $3G$ is finite.

\end{myromanlist}
\end{Ltheorem}
\end{samepage}

\pagebreak[2]

Theorem~\ref{thm:qcs:main} answers a question of L. Aussenhofer in the 
case of compact abelian groups. The proof of Theorem~\ref{thm:qcs:main} is 
presented in \S\ref{sect:qcs}. One of its main ingredients is the next 
theorem, which (together with the results mentioned in 
Example~\ref{ex:intro:TJ23R}) implies that for \mbox{every} prime $p$,~the 
compact group $\mathbb{J}_p$ of $p$-adic integers admits a non-trivial 
quasi-convex null sequence.

\begin{Ltheorem} \label{thm:Jp:main}
Let $p \geq 5$ be a prime, 
$\underline a=\{a_n\}_{n=0}^\infty$ an increasing sequence of 
non-negative integers, and put 
\mbox{$y_n = p^{a_n}$}. The set
$L_{\underline a,p} = \{0\} \cup \{\pm y_n \mid n \in \mathbb{N}\}$
is quasi-convex in~$\mathbb{J}_p$.
\end{Ltheorem}

The proof of Theorem~\ref{thm:Jp:main} is presented in \S\ref{sect:Jp}. 
Theorem~\ref{thm:Jp:main} does not hold for \mbox{$p \not\geq 5$}
(see Example~\ref{ex:Jp:p23}). Dikranjan and de Leo gave sufficient 
conditions on the  sequence $\underline a$ to ensure that 
$L_{\underline a, 2}$ is 
quasi-convex in $\mathbb{J}_2$ (cf.~\cite[1.4]{DikLeo}). For the case 
where $p=3$, Dikranjan and Luk\'acs gave a complete characterization of 
those sequences $\underline a$ such that $L_{\underline a, 3}$ is 
quasi-convex (cf.~\cite[Theorem~D]{DikGL1}).

The arguments and techniques developed in the proof of Theorem B are 
applicable {\em mutandi mutandis} for sequences with $p$-power  denominators in $\mathbb{T}$. This is done in the next 
theorem, whose proof is given in \S\ref{sect:Tp}. Theorem~\ref{thm:Tp:main}
is a continuation of our study  of non-trivial  quasi-convex null 
sequences in $\mathbb{T}$ (cf.~\cite{DikGL1}).

\begin{Ltheorem} \label{thm:Tp:main}
Let $p \geq 5$ be a prime, 
$\underline a=\{a_n\}_{n=0}^\infty$ an increasing sequence of 
non-negative integers, and put 
\mbox{$x_n = p^{-(a_n+1)}$}. The set
$K_{\underline a,p} = \{0\} \cup \{\pm x_n \mid n \in \mathbb{N}\}$
is quasi-convex in~$\mathbb{T}$.
\end{Ltheorem}

Theorem~\ref{thm:Tp:main} does not hold for \mbox{$p \not\geq 5$} (see 
Example~\ref{ex:Tp:p23}). Dikranjan and Luk\'acs gave complete 
characterizations of those sequences $\underline a$ such that 
$K_{\underline a, 2}$ and $K_{\underline a, 3}$ are quasi-convex in 
$\mathbb{T}$ (cf.~\cite[Theorem~A, Theorem~C]{DikGL1}).

\medskip

We are leaving open the following three problems. The first two are 
motivated by the observation that certain parts of 
Theorem~\ref{thm:qcs:main} 
hold for a class larger than that of (locally) compact abelian groups. 
Indeed, one can prove, for instance, that items (i), (ii), and (v) of 
Theorem~\ref{thm:qcs:main} remain equivalent for so-called  
$\omega$-bounded abelian groups. (A topological group $G$ is called 
{\em $\omega$-bounded} if every countable subset of $G$ is contained in
some compact subgroup of $G$.)

\begin{problem}
Is it possible to replace the class of locally compact abelian groups  in
Theorem~\ref{thm:qcs:main} with a~different class of abelian topological
groups that contains all compact abelian groups?
\end{problem}

\begin{problem}
Is it possible to obtain a characterization for countably compact abelian 
groups that admit no non-trivial quasi-convex null sequences in terms of 
their structures, in the spirit of Theorem~\ref{thm:qcs:main}?
\end{problem}


\begin{problem}
Let $H$ be an infinite cyclic subgroup of $\mathbb{T}$. Does $H$
admit a non-trivial quasi-convex  null sequence?
\end{problem}

\section{Preliminaries: Exotic tori and abelian pro-finite groups}

\label{sect:prel}

In this section, we provide a few well-known definitions and results that 
we rely on later. We chose to isolate these in order to improve the 
flow of arguments in \S\ref{sect:qcs}.

\begin{definition} (\cite{DikProET}, \cite[p.~141]{DikProSto})
A compact abelian group is an {\em exotic torus} if it
contains no subgroup that is topologically isomorphic to 
$\mathbb{J}_p$ for some prime $p$.
\end{definition}

The notion of exotic torus was introduced by Dikranjan and Prodanov in 
\cite{DikProET}, who also provided, among other things, the following 
characterization for such groups.

\pagebreak[2]

\begin{ftheorem}[{\cite{DikProET}}] \label{thm:prel:ET}
A compact abelian group $K$ is an 
exotic torus if and only if it contains a~closed subgroup $B$ such that

\begin{myromanlist}

\item
$K/B \cong \mathbb{T}^n$ for some $n \in \mathbb{N}$, and

\item
$B = \prod\limits_{p} B_p$, where each $B_p$ is a compact bounded abelian 
$p$-group.

\vspace{6pt}

\end{myromanlist}
Furthermore, if $K$ is connected, then each $B_p$ is finite.
\end{ftheorem}

\pagebreak[2]

Since \cite{DikProET} is not easily accessible for most readers, we 
provide a sketch of the proof of Theorem~\ref{thm:prel:ET} for the sake of 
completeness. The main idea of the proof is to pass to the Pontryagin 
dual, and then to establish that an abelian group is the  dual of an 
exotic  torus if and only of it satisfies conditions that are dual to (i) 
and (ii). To that end, we recall that an abelian group $X$ is said to 
be {\em strongly non-divisible} if there is no surjective homomorphism
\mbox{$X \rightarrow \mathbb{Z}(p^\infty)$} for any prime $p$ 
(cf.~\cite{DikProET}).

\begin{proof}
It follows from Pontryagin duality that a compact abelian group $K$ 
contains no subgroup that is topologically isomorphic to $\mathbb{J}_p$ 
(i.e., $K$ is an exotic torus) if and only if its Pontryagin dual 
\mbox{$X=\widehat K$} has no quotient that is isomorphic to 
\mbox{$\mathbb{Z}(p^\infty)\cong \widehat{\mathbb{J}}_p$}
(cf.~\cite[Theorem~54]{Pontr}), that is, $X$ is strongly non-divisible.
Similarly, by 
Pontryagin duality (cf.~\cite[Theorem~37 and~54]{Pontr}), a compact 
abelian group $K$ satisfies  (i) and (ii) if and only if its Pontryagin 
dual \mbox{$X=\widehat K$} contains a subgroup $F$ such that
\begin{myromanlist}

\item[(i$^\prime$)]
$F \cong \mathbb{Z}^n$  for some $n \in \mathbb{N}$, and

\item[(ii$^\prime$)]
$X/F \cong \bigoplus\limits_p T_p$, where each $T_p$ is a bounded abelian 
$p$-group.

\vspace{6pt}

\end{myromanlist}
Hence, we proceed by showing that a discrete abelian group $X$ is strongly 
non-divisible if and only if $X$ contains a subgroup $F$ that satisfies 
(i$^\prime$) and (ii$^\prime$).

\pagebreak[3]

Suppose that $X$ is strongly non-divisible. If $X$ contains a free abelian 
group $F_0$ of infinite rank, then there is a surjective homomorphism
\mbox{$f\colon F_0 \rightarrow \mathbb{Z}(2^\infty)$}, and
$f$ extends to a surjective homomorphism
\mbox{$\bar f \colon X \rightarrow \mathbb{Z}(2^\infty)$}, because
$\mathbb{Z}(2^\infty)$ is divisible. Since $X$ is strongly non-divisible, 
this is impossible. Thus, $X$ contains a free abelian subgroup $F$ of 
finite rank $n$ such that $X/F$ is a~torsion group (i.e., $r_0(X)=n$).
Clearly, $F\cong \mathbb{Z}^n$, and so (i$^\prime$) holds. Let 
\mbox{$X/F = \bigoplus\limits_p T_p$} be  the primary decomposition of $X/F$
(cf.~\cite[8.4]{Fuchs}), fix a prime $p$, and let $A_p$ be a subgroup of 
$T_p$ such that $A_p$ is a direct sum of cyclic subgroups, and $T_p/A_p$ 
is divisible (e.g., take a $p$-basic subgroup of~$T_p$; 
cf.~\cite[32.3]{Fuchs}). One has
$T_p/A_p \cong \bigoplus\limits_{\kappa_p} \mathbb{Z}(p^\infty)$ 
(cf.~\cite[23.1]{Fuchs}). Since $T_p/A_p$ is a~homomorphic image of the 
strongly non-divisible group $X$, this implies that $\kappa_p=0$, and
so $T_p=A_p$. Therefore, $T_p$ is a direct sum of cyclic groups. Finally, 
if $T_p$ is not bounded, then it admits a surjective homomorphism
\mbox{$T_p \rightarrow \mathbb{Z}(p^\infty)$}. This, however, is 
impossible, because $T_p$ is a homomorphic image of $X$. Hence,
each $T_p$ is a bounded $p$-group, and (ii$^\prime$) holds.

Conversely, suppose that $X$ contains a subgroup $F$ that satisfies
(i$^\prime$) and (ii$^\prime$). Assume that $X$ is not strongly 
non-divisible, that is, there is a prime $p$ such that there exists a 
surjective homomorphism \mbox{$f\colon X \rightarrow \mathbb{Z}(p^\infty)$}.
Since $\mathbb{Z}(p^\infty)$ is not finitely generated,  the finitely 
generated image  $f(F)$ is a proper subgroup of $\mathbb{Z}(p^\infty)$,
and so \mbox{$\mathbb{Z}(p^\infty)/f(F)\cong \mathbb{Z}(p^\infty)$}.
Thus, by replacing $f$ with its composition with the canonical 
projection 
\mbox{$\mathbb{Z}(p^\infty)\rightarrow \mathbb{Z}(p^\infty)/f(F)$}, we may 
assume that \mbox{$F \hspace{-2pt}\subseteq\hspace{-1pt} \ker f$}. 
Therefore, $f$ induces  a surjective homomorphism 
\mbox{$g\colon X/F \rightarrow \mathbb{Z}(p^\infty)$}. Since 
$g(\bigoplus\limits_{q\neq p} T_q) = \{0\}$, surjectivity of $g$ implies 
that $g_{|T_p}$ is surjective. This, however, is a contradiction,
because by (ii$^\prime$),~$T_p$~is bounded. Hence, $X$ is strongly 
non-divisible.

It remains to be seen that if $K$ is a connected exotic tori, then each 
$B_p$ in (ii) is finite. In terms of the Pontryagin dual 
\mbox{$X=\widehat K$}, this is 
equivalent to the statement that if $X$ satisfies (i$^\prime$) and  
(ii$^\prime$), and $X$ is torsion free, then each $T_p$ 
is finite (cf.~\cite[Theorem~46]{Pontr}). We proceed by showing this 
latter implication. By (ii$^\prime)$, $X$ contains a~subgroup $F$ such 
that  \mbox{$F\cong \mathbb{Z}^n$} for some 
\mbox{$n \in \mathbb{N}$}. We have already seen that there is a~maximal 
$n$ with respect to this property. Since $X$ is torsion free, maximality 
of $n$ 
implies that $F$ meets non-trivially every non-zero subgroup of $X$ (i.e., 
$F$ is an {\em essential} subgroup). Let 
\mbox{$i\colon F \rightarrow \mathbb{Q}^n$} be a~monomorphism such that
\mbox{$i(F)=\mathbb{Z}^n$}. Since $\mathbb{Q}^n$ is divisible, $i$ 
can be extended
to a~homomorphism \mbox{$j\colon X \rightarrow \mathbb{Q}^n$}, and $j$ is 
a~monomorphism, because 
\mbox{$F \hspace{-2pt} \cap\hspace{-1.5pt} \ker j = \ker i$} is trivial 
(and $F$ is an essential subgroup).
Therefore, $j$ induces a~monomorphism  
\mbox{$\bar j \colon X/F\rightarrow \mathbb{Q}^n/\mathbb{Z}^n$}.
Since  \mbox{$\mathbb{Q}^n/\mathbb{Z}^n\cong 
\bigoplus\limits_p \mathbb{Z}(p^\infty)^n$}, the image
$\bar j(T_p)$ is isomorphic to a~bounded subgroup of 
$\mathbb{Z}(p^\infty)^n$ for every prime $p$. Hence, each $T_p$ is finite, 
because all bounded subgroups of $\mathbb{Z}(p^\infty)^n$ are finite.
This completes the proof.
\end{proof}

Recall that a topological group is {\em pro-finite} if it is the 
(projective) limit of finite groups, or equivalently, if it is compact and 
zero-dimensional. For a prime $p$, a topological group $G$ is called a 
{\em pro-$p$-group} if it is the (projective) limit of finite $p$-groups, 
or equivalently, if it is pro-finite and $x^{p^n} \longrightarrow e$ 
for every $x \in G$ (or,  in the abelian case, $p^n x \longrightarrow 0$). 

\begin{ftheorem}[{\cite{ArmacostLCA}, 
\cite[Corollary 8.8(ii)]{HofMor}, \cite[4.1.3]{DikProSto}}]
\label{thm:prel:profin}
Let $G$ be a pro-finite group. Then 
\mbox{$G\hspace{-2pt} =\hspace{-2pt} \prod\limits_{p} \hspace{-2pt} G_p$}, 
where each $G_p$ is a pro-$p$-group.
\end{ftheorem}

\section{LCA groups that admit a non-trivial quasi-convex null sequence}
\label{sect:qcs}

In this section, we prove Theorem~\ref{thm:qcs:main} by using 
two intermediate steps: First, we consider direct products of finite 
cyclic groups, and then we show that Theorem~\ref{thm:qcs:main} holds for 
pro-finite groups. We start off with a lemma that allows us to
relate non-trivial quasi-convex null sequences in closed (and open) 
subgroups to those in the ambient group.

\begin{lemma} \label{lemma:qcs:subgr}
Let $G$ be a locally compact abelian group, and $H$ a closed subgroup.

\begin{myalphlist}

\item
If $H$ admits a non-trivial quasi-convex null sequence, then so does $G$.

\item
If $H$ is open in $G$, then $H$ admits a non-trivial quasi-convex null 
sequence if and only if $G$ does.

\end{myalphlist}
\end{lemma}

In order to prove Lemma~\ref{lemma:qcs:subgr},  we rely on
the following general property of the quasi-convexity hull, which will 
also be used in the proof of Theorem~\ref{thm:qcs:cyclic} below.

\begin{fproposition}[{\cite[I.3(e)]{GLdualtheo}, \cite[2.7]{DikLeo}}]
\label{prop:qcs:qc-hom}
If $f\colon G\to H$ is a continuous homomorphism of abelian
topological groups, and $E\subseteq G$, then $f(Q_G(E))\subseteq 
Q_H(f(E))$.
\end{fproposition}

\begin{proof}[Proof of Lemma~\ref{lemma:qcs:subgr}.]
(a) Let \mbox{$\iota\colon H \rightarrow G$} denote the inclusion. 
By Proposition~\ref{prop:qcs:qc-hom}, 
\begin{align}
Q_H(S) \subseteq \iota^{-1}(Q_G(S))=Q_G(S)\cap H.
\end{align}
On the other hand, since $H$ is a subgroup, $H^\triangleright = H^\perp$,
where  \mbox{$H^\perp = \{\chi \hspace{-2pt}\in\hspace{-2pt} \widehat G
\mid\chi(H)=\{0\}\}$} is the annihilator of $H$ in $\widehat G$.
Thus, by Pontryagin duality (cf.~\cite[Theorems~37 and~54]{Pontr}),
\begin{align}
Q_G(H)=(H^\perp)^\triangleright = H^{\perp\perp}=H,
\end{align}
and so \mbox{$Q_G(S) \subseteq Q_G(H)=H$}. Finally, 
\mbox{$Q_G(S) \subseteq Q_H(S)$},
because by Pontryagin duality, every character of $H$ extends to a
character of $G$ (cf.~\cite[Theorem~54]{Pontr}).

\pagebreak[2]

(b) The necessity of the condition follows from (a). In order to show 
sufficiency, let $\{x_n\}_{n=0}^\infty$ be a non-trivial quasi-convex null 
sequence in $G$, and put $S=\{0\}\cup \{\pm x_n \mid n \in \mathbb{N}\}$.
Let \mbox{$\iota\colon H \rightarrow G$} denote the inclusion. Then,
by Proposition~\ref{prop:qcs:qc-hom}, \mbox{$\iota^{-1}(S)=S\cap H$} is 
quasi-convex in $H$. Since $H$ is a 
neighborhood of $0$ and $\{x_n\}_{n=0}^\infty$ is a non-trivial null 
sequence, the intersection \mbox{$S\cap H$} is infinite. Therefore, the 
subsequence $\{x_{n_k}\}_{k=0}^\infty$ of  $\{x_n\}_{n=0}^\infty$ 
consisting of the members that belong to $H$ is a~non-trivial quasi-convex 
null sequence in $H$.
\end{proof}

\begin{lemma} \label{lemma:qcs:Z23}
Let $G$ be an abelian topological group of exponent $2$ or $3$. Then $G$ 
admits no \mbox{non-trivial} quasi-convex null sequences.  In particular, 
the groups $\mathbb{Z}_2^\kappa$ and $\mathbb{Z}_3^\kappa$ admit no 
non-trivial quasi-convex null sequences for any cardinal $\kappa$.
\end{lemma}

\begin{proof}
Observe that if $x$ is an element of order $2$ or $3$ in an
abelian topological group~$G$,~and 
\mbox{$\chi(x) \in \mathbb{T}_+$}, then \mbox{$\chi(x)=0$}, and so 
\mbox{$\chi \in \langle x \rangle^\perp$}. Thus, if 
\mbox{$S \subseteq G$} and $S$ consists of elements of order at 
most~$3$,~then
\mbox{$S^\triangleright = S^\perp=\langle S\rangle^\perp$} is a closed 
subgroup of $\widehat G$. Consequently, 
$Q_G(S)=(\langle S\rangle^{\perp})^\perp$ is a~subgroup of $G$.
Therefore, if $S$ is a non-trivial null sequence, then 
$S \subsetneq Q_G(S)$, because every group is homogeneous. 
Hence, no abelian topological group contains non-trivial quasi-convex 
sequences consisting of elements of order at most~$3$. Since each element
in $\mathbb{Z}_2^\kappa$ and $\mathbb{Z}_3^\kappa$ has order at most 
$3$, this completes the proof.
\end{proof}

Lemma~\ref{lemma:qcs:Z23} combined with Lemma~\ref{lemma:qcs:subgr}(b) 
yields the following consequence.

\begin{corollary} \label{cor:qcs:Z23}  
Let $G$ be a locally compact abelian group, and $\kappa$ a cardinal. If 
$G$ contains an open subgroup that is topologically isomorphic to
$\mathbb{Z}_2^\kappa$ or $\mathbb{Z}_3^\kappa$, then $G$ admits no 
non-trivial quasi-convex null sequences. \qed
\end{corollary}

\begin{theorem} \label{thm:qcs:cyclic}
Let $\{m_k\}_{k=1}^\infty$ be a sequence of integers such that
\mbox{$m_k \geq 4$} \ for every 
\mbox{$k\in\mathbb{N}$}. Then the product
$P=\prod\limits_{k=0}^\infty \mathbb{Z}_{m_k}$ admits a non-trivial 
quasi-convex null sequence.
\end{theorem}

\begin{proof}
Let $\pi_k\colon P \rightarrow \mathbb{Z}_{m_k}$ denote the canonical 
projection for each $k \in \mathbb{N}$, and let $e_n \in P$ be such that
$\pi_k(e_n)=0$ if $k\neq n$, and $\pi_k(e_k)$ generates 
$\mathbb{Z}_{m_k}$. Clearly,  $\{e_n\}_{n=0}^\infty$ is a non-trivial null 
sequence, and so it remains to be seen that it is quasi-convex. To that 
end, put $S=\{0\} \cup \{\pm e_n\mid n \in \mathbb{N}\}$. Since
$\pi_k(S) = \{0,\pm \pi_k(e_k)\}$ is quasi-convex in 
$\mathbb{Z}_{m_k}$ for every $k \in \mathbb{N}$ 
(cf.~\cite[7.8]{Aussenhofer}), by 
Proposition~\ref{prop:qcs:qc-hom},
\begin{align} \label{eq:qcs:QPS}
Q_P(S) \subseteq \bigcap\limits_{k\in \mathbb{N}} 
\pi_k^{-1}(Q_{\mathbb{Z}_{m_k}}(\pi_k(S))) = 
\bigcap\limits_{k\in \mathbb{N}} \pi_k^{-1}(\{0,\pm \pi_k(e_k)\}) = 
\prod\limits_{k\in \mathbb{N}} \{0,\pm \pi_k(e_k)\}.
\end{align}
Every element in the compact group $P$ can be expressed in the form
\mbox{$x=\sum\limits_{k \in \mathbb{N}} c_k e_k$}, where 
\mbox{$c_k \in \mathbb{Z}_{m_k}$} for every \mbox{$k \in \mathbb{N}$}. 
For each \mbox{$k \in \mathbb{N}$}, let 
\mbox{$\chi_k\colon P \rightarrow \mathbb{T}$} denote the 
continuous character 
defined by \mbox{$\chi_k(x) = \frac{c_k}{m_k} + \mathbb{Z}$}, and put
$l_k = \lfloor \frac {m_k} 4 \rfloor$. (As $\chi_k$ factors through 
$\pi_k$, it is indeed continuous.) Then
\begin{align}
l_k\chi_k(e_n) = \begin{cases}
\frac{l_k}{m_k} & n=k\\
0 & n \neq k.
\end{cases}
\end{align}
Consequently, \mbox{$(l_{k_1} \chi_{k_1} \pm l_{k_2} \chi_{k_2})(e_n) \in 
\mathbb{T}_+$} for every \mbox{$k_1 \neq k_2$}, and thus
\mbox{$l_{k_1} \chi_{k_1} \pm l_{k_2} \chi_{k_2} \in S^\triangleright$}.
Let \mbox{$x \in Q_P(S)\backslash\{0\}$}. By (\ref{eq:qcs:QPS}), 
\mbox{$x=\sum\limits_{k \in \mathbb{N}} c_k e_k$}, where 
\mbox{$c_k\in \{0,1,-1\}$}. Let \mbox{$k_1 \in \mathbb{N}$} be the 
smallest index such that \mbox{$c_{k_1} \neq 0$}. By replacing $x$ with 
$-x$ if necessary, we may assume that \mbox{$c_{k_1}=1$}. Let 
$k_2 \in \mathbb{N}$ be such that $k_2 \neq k_1$. By what we have shown so 
far, $l_{k_1}\chi_{k_1} + c_{k_2}l_{k_2}\chi_{k_1} \in S^\triangleright$.
Therefore,
\begin{align} \label{eq:qcs:k12}
\tfrac{l_{k_1}}{m_{k_1}} + \tfrac{l_{k_2} c_{k_2}^2}{m_{k_2}} + \mathbb{Z}
= (l_{k_1}\chi_{k_1} + c_{k_2}l_{k_2}\chi_{k_1})(x) \in \mathbb{T}_+.
\end{align}
Since $m_k \geq 4$ for all $k \in \mathbb{N}$, one has $l_k >0$, and so
(\ref{eq:qcs:k12}) implies that $c_{k_2}=0$. Hence, $c_k = 0$ for all 
$k \neq k_1$. This shows that $x \in S$,  and $S$ is 
quasi-convex, as desired.
\end{proof}

\begin{corollary}
\label{cor:qcs:nosub}
Let $G$ be a locally compact abelian group that admits no 
non-trivial 
quasi-convex null sequences. Then $G$ has no subgroups that are
topologically isomorphic to: 

\begin{myalphlist}

\item
$\mathbb{J}_p$ for some prime $p$,

\item
$\mathbb{Z}_{p^2}^\omega$ \ for some prime $p$, 

\item
$\mathbb{Z}_p^\omega$ \ for $p > 3$,

\item
$\mathbb{T}$, or

\item
$\mathbb{R}$.

\end{myalphlist}
\end{corollary}

\begin{proof}
By Example~\ref{ex:intro:TJ23R} and Theorem~\ref{thm:Jp:main},
for every prime $p$, the group $\mathbb{J}_p$ admits a non-trivial 
quasi-convex null sequence. By Theorem~\ref{thm:qcs:cyclic},
for every prime $p$, the countable product
$\mathbb{Z}_{p^2}^\omega$ admits a non-trivial 
quasi-convex null sequence, and if \mbox{$p > 3$}, then so does the group
$\mathbb{Z}_{p}^\omega$. Finally,  by Example~\ref{ex:intro:TJ23R},
$\mathbb{T}$ and $\mathbb{R}$ admit a non-trivial quasi-convex
null sequence. Hence, all five statements follow by 
Lemma~\ref{lemma:qcs:subgr}(a).
\end{proof}

\begin{corollary} \label{cor:qcs:pro-p}
Let $p$ be a prime, and $G$ an abelian pro-$p$-group. Then $G$ 
admits no non-trivial quasi-convex null sequences if and only if 

\begin{myromanlist}

\item
$p> 3$ and $G$ is finite, or

\item
$p \leq 3$ and
$G \cong \mathbb{Z}_p^\kappa \times F$
for some cardinal $\kappa$ and finite group $F$.

\end{myromanlist}
\end{corollary}

\begin{proof}
Suppose that $G$ admits no non-trivial quasi-convex null sequences. Then, 
by Corollary~\ref{cor:qcs:nosub}(a), $G$ contains no subgroup that is 
topologically isomorphic to $\mathbb{J}_p$, and so $G$ is an exotic torus.
Let $B$ be a 
closed subgroup of $G$ provided by  Theorem~\ref{thm:prel:ET}. Since $G$ 
is a pro-$p$-group, it  has no connected quotients. So, $n=0$, and 
$G=B$. Thus, $G=B_p$ is a compact bounded $p$-group (as $G$ contains 
no elements of order coprime to $p$). Consequently, $G$ is topologically 
isomorphic to a~product of finite cyclic groups 
(cf.~\cite[4.2.2]{DikProSto}).
Hence, $G\cong \prod\limits_{i=1}^N \mathbb{Z}_{p^i}^{\kappa_i}$ for
some cardinals $\{\kappa_i\}_{i=1}^N$.~By 
Corollary~\ref{cor:qcs:nosub}(b), $G$ contains no subgroup that is
topologically isomorphic to $\mathbb{Z}_{p^2}^\omega$.
Therefore, $\kappa_i$~is finite for \mbox{$i \geq 2$}, and 
$G \cong \mathbb{Z}_p^{\kappa_1} \times F$ for some 
finite group $F$. Hence, by Corollary~\ref{cor:qcs:nosub}(c), 
$\kappa < \omega$  or $p \leq 3$, as desired.

Conversely, if $G$ is finite, then it contains no non-trivial sequences, 
and if (ii) holds, then $G$ contains an open subgroup $H$ that is
topologically isomorphic to $\mathbb{Z}_2^\kappa$ or $\mathbb{Z}_3^\kappa$, 
so the statement follows by Corollary~\ref{cor:qcs:Z23}.
\end{proof}

\begin{proposition} \label{prop:qcs:profin}
Let $G$ be an abelian pro-finite group. Then $G$
admits no non-trivial quasi-convex null sequences if and only if
$G\cong \mathbb{Z}_2^\kappa \times F$ or 
$G\cong \mathbb{Z}_3^\kappa \times F$
for some cardinal $\kappa$ and finite abelian group $F$.
\end{proposition}

\begin{proof}
Suppose that $G$ admits no non-trivial quasi-convex null sequences.
By Lemma~\ref{lemma:qcs:subgr}(a), no closed subgroup of $G$ admits a
non-trivial quasi-convex null sequence. By Theorem~\ref{thm:prel:profin}, 
\mbox{$G\hspace{-1pt}= \hspace{-1pt}\prod\limits_{p} G_p$}, where each 
$G_p$ is a pro-$p$-group. Since each $G_p$ is a closed subgroup, $G_p$ 
admits no non-trivial quasi-convex null sequences. Therefore, by 
Corollary~\ref{cor:qcs:pro-p}, for each  \mbox{$p>3$}, the subgroup $G_p$ 
is finite. Put 
\mbox{$F_0 = \prod\limits_{p>3} G_p$}. If $F_0$ is infinite, then 
there are infinitely many primes 
\mbox{$p_k > 3$} such that 
\mbox{$G_{p_k} \neq 0$}. Consequently, $F_0$ (and thus $G$) contains a 
subgroup that is topologically isomorphic to the product
\mbox{$P=\prod\limits_{k=1}^\infty \mathbb{Z}_{p_k}$}. However, by 
Theorem~\ref{thm:qcs:cyclic}, $P$ does admit a non-trivial quasi-convex 
null sequence, contrary to our assumption
(and Lemma~\ref{lemma:qcs:subgr}(a)). This contradiction shows that $F_0$ 
is finite.
Since $G_2$ and $G_3$ contain no non-trivial quasi-convex null sequences, 
by Corollary~\ref{cor:qcs:pro-p}, $G_2\cong\mathbb{Z}_2^{\kappa_2}\times F_2$ 
and $G_3 \cong \mathbb{Z}_3^{\kappa_3} \times F_3$ for some cardinals
$\kappa_2$ and $\kappa_3$, and finite groups $F_2$ and $F_3$. Thus,
\begin{align}
G = \prod\limits_{p} G_p = G_2 \times G_3 \times F_0 \cong
\mathbb{Z}_2^{\kappa_2} \times  \mathbb{Z}_3^{\kappa_3} \times F_0\times 
F_2 \times F_3,
\end{align}
and it remains to be seen that at least one of $\kappa_2$ and $\kappa_3$ 
is finite. Assume the contrary. Then $G$ contains a (closed) subgroup that 
is topologically isomorphic to 
\mbox{$\mathbb{Z}_2^\omega \times \mathbb{Z}_3^\omega \cong 
\mathbb{Z}_6^\omega$}, which does admit a non-trivial quasi-convex null 
sequence by Theorem~\ref{thm:qcs:cyclic}, contrary to our assumption
(and Lemma~\ref{lemma:qcs:subgr}(a)). Therefore, at least one of 
$\mathbb{Z}_2^{\kappa_2} \times F_0\times F_2 \times F_3$ and
$\mathbb{Z}_3^{\kappa_3} \times F_0\times F_2 \times F_3$ is finite.

Conversely,  if $G\cong \mathbb{Z}_2^\kappa \times F$ or
$G\cong \mathbb{Z}_3^\kappa \times F$ where $F$ is finite, then $G$ 
contains an open subgroup $H$ that is
topologically isomorphic to $\mathbb{Z}_2^\kappa$ or $\mathbb{Z}_3^\kappa$,
so the statement follows by Corollary~\ref{cor:qcs:Z23}.
\end{proof}

\begin{proof}[Proof of Theorem~\ref{thm:qcs:main}.]
We first consider the special case where $G$ is compact. Clearly, in this 
case, (ii) $\Leftrightarrow$ (v) and (iii) $\Leftrightarrow$ (iv).

(i) $\Rightarrow$ (v): Suppose that $G$ admits no non-trivial
quasi-convex null sequences. Let $K$ denote the connected component of 
$G$. By Lemma~\ref{lemma:qcs:subgr}(a), $K$ admits no non-trivial
quasi-convex null sequences. Thus, by Corollary~\ref{cor:qcs:nosub}(a), 
$K$ is an exotic torus, and by Theorem~\ref{thm:prel:ET}, $K$ contains 
a~closed subgroup $B$ such that $B = \prod\limits_{p} B_p$, where each 
$B_p$ is a finite $p$-group, and $K/B \cong \mathbb{T}^n$ for some 
$n \in \mathbb{T}$. The group $B$ is pro-finite, and by 
Lemma~\ref{lemma:qcs:subgr}(a), it admits no non-trivial
quasi-convex null sequences. Therefore, by 
Proposition~\ref{prop:qcs:profin}, $B\cong \mathbb{Z}_2^\kappa \times F$
or $B\cong \mathbb{Z}_3^\kappa \times F$,
where $\kappa$ is some cardinal, and $F$ is a finite abelian group.
Since $B_2$ and $B_3$ are finite, this implies that $B$ itself is finite.
Consequently, by Pontryagin duality, 
$\widehat B \cong \widehat K / B^\perp$ is finite 
(cf.~\cite[Theorem~54]{Pontr}), and 
$B^\perp \cong \widehat{K/B} = \mathbb{Z}^n$
(cf.~\cite[Theorem~37]{Pontr}). This implies that $\widehat K$ is finitely 
generated. On the other hand, $\widehat K$ is torsion free, because $K$ is 
connected (cf.~\cite[Example~73]{Pontr}), which means that
$\widehat K = \mathbb{Z}^n$ and $K \cong \mathbb{T}^n$. By 
Corollary~\ref{cor:qcs:nosub}(d), $K$ contains no subgroup that is 
topologically isomorphic to $\mathbb{T}$. Hence, $n=0$, and $K=0$.
This shows that (i) implies that $G$ is pro-finite (i.e., 
a zero-dimensional compact group). The statement follows now by 
Proposition~\ref{prop:qcs:profin}.

(v) $\Rightarrow$ (iv) is obvious, because (v) implies that $G$ is 
a compact bounded group; therefore,~$G$~is a direct product of finite 
cyclic  groups (cf.~\cite[4.2.2]{DikProSto}).

(iv) $\Rightarrow$ (i): It follows from (iv) that $G$ is pro-finite, and 
thus the statement is a consequence of Proposition~\ref{prop:qcs:profin}.

\pagebreak[2]

Suppose now that $G$ is a locally compact abelian group.

(i) $\Rightarrow$ (ii): Suppose that $G$ admits no non-trivial 
quasi-convex null sequences. There is \mbox{$n\in\mathbb{N}$} and a closed 
subgroup $M$ such that $M$ has an open compact subgroup $K$, and
\mbox{$G \cong \mathbb{R}^n \times M$} (cf.~\cite[3.3.10]{DikProSto}). 
By Corollary~\ref{cor:qcs:nosub}(e), \mbox{$n=0$}, and so 
\mbox{$G=M$} and $K$ is open 
in $G$. By Lemma~\ref{lemma:qcs:subgr}(a), $K$ admits no non-trivial
quasi-convex null sequences. Thus, by what we have shown so far,
$K$ has an open compact subgroup $O$ that is topologically isomorphic to
$\mathbb{Z}_2^\kappa$ or $\mathbb{Z}_3^\kappa$ for some cardinal $\kappa$.
Since $K$ is open in $G$, it follows that $O$ is open in $G$. Therefore,
one of $G[2]$ and $G[3]$ is open in $G$, because $O \subseteq G[2]$ or 
$O \subseteq G[3]$.

(ii) $\Rightarrow$ (iii): Let $L$ be a bounded locally compact abelian 
group. Then there   is $n\in\mathbb{N}$ and a~closed subgroup 
$M$ such that $M$ has an open compact subgroup $K$, and
\mbox{$L \cong \mathbb{R}^n \times M$} (cf.~\cite[3.3.10]{DikProSto}).
Since $L$ is bounded, 
$n=0$, and thus $L$ admits an open compact bounded subgroup $K$. 
Consequently, $K$ is a direct product of finite cyclic groups 
(cf.~\cite[4.2.2]{DikProSto}). By this argument, for every locally compact 
abelian group $G$, the subgroups $G[2]$ and $G[3]$ contain open subgroups 
$O_2$ and $O_3$, respectively, such that
$O_2 \cong \mathbb{Z}_2^{\kappa_2}$ and  
$O_3 \cong \mathbb{Z}_3^{\kappa_3}$
for some cardinals  $\kappa_2$  and $\kappa_3$. If one of $G[2]$ and 
$G[3]$ is open in $G$, then $O_2$ or $O_3$ is open in $G$, as desired. 

(iii) $\Rightarrow$ (i) follows by Corollary~\ref{cor:qcs:Z23}.
\end{proof}

\section{Sequences of the form $\boldsymbol{\{p^{a_n}\}_{n=1}^\infty}$ in 
$\boldsymbol{\mathbb{J}_p}$}

\label{sect:Jp}

In this section, we present the proof of Theorem~\ref{thm:Jp:main}. 
We start off by establishing a few preliminary facts concerning 
representations of elements 
in $\mathbb{T}$, which are recycled and reused in \S\ref{sect:Tp}, 
where Theorem~\ref{thm:Tp:main} is proven.
We identify points of $\mathbb{T}$ with $(-1/2,1/2]$. Let 
\mbox{$p>2$} be a 
prime. Recall that every
\mbox{$y \in (-1/2,1/2]$} can be written in the form
\begin{align} \label{eq:Jp:repy}
y=\sum\limits_{i=1}^\infty \frac{c_i}{p^i} = 
\frac{c_1}{p} + \frac{c_2}{p^2} + \cdots + \frac{c_s}{p^s}+\cdots,
\end{align}
where $c_i\in\mathbb{Z}$ and $|c_i| \leq \frac{p-1}{2}$ for all
$i \in \mathbb{N}$.

\begin{theorem} \label{thm:Jp:c1}
Let \mbox{$p>2$} be a prime, and 
\mbox{$y=\sum\limits_{i=1}^\infty \frac{c_i}{p^i} \in \mathbb{T}$},
where  \mbox{$c_i\in\mathbb{Z}$} and 
\mbox{$|c_i| \leq \frac{p-1}{2}$}
for all \mbox{$i \in \mathbb{N}$}.
\begin{myalphlist}

\item
If $y \in \mathbb{T}_+$, then $|c_1| \leq \lfloor\tfrac{p+2}{4} \rfloor$.

\item
If $my \in \mathbb{T}_+$ for all $m=1,\ldots, \lceil\tfrac{p}{2} \rceil$,
then $c_1=0$.

\item
If $my \in \mathbb{T}_+$ for all $m=1,\ldots, \lceil\tfrac{p}{6} \rceil$,
then $c_1\in \{-1,0,1\}$.

\end{myalphlist}
\end{theorem}

\begin{proof}
Since $|c_i| \leq \tfrac{p-1}{2}$ for all $i \in \mathbb{N}$, one has
\begin{align} \label{eq:Jp:estimate}
\left|\sum\limits_{i=2}^\infty \dfrac{c_i}{p^i} \right| \leq
\sum\limits_{i=2}^\infty \dfrac{|c_i|}{p^i} \leq
\dfrac{p-1}{2}\left(\sum\limits_{i=2}^\infty \dfrac{1}{p^i}\right) = 
\dfrac{p-1}{2} \cdot \dfrac{1}{p^2}\cdot \dfrac{p}{p-1}=\dfrac{1}{2p}.
\end{align}

(a) If $y \in \mathbb{T}_+$, then by (\ref{eq:Jp:estimate}), one obtains 
that
\begin{align}
\left| \dfrac{c_1}{p} \right| \leq 
|y| + \left|\sum\limits_{i=2}^\infty \dfrac{c_i}{p^i} \right|  \leq
\dfrac 1 4 + \dfrac 1 {2p} = \dfrac{p+2}{4p}.
\end{align}
Therefore, $|c_1| \leq \tfrac{p+2}{4}$, and hence
$|c_1| \leq \lfloor \tfrac{p+2}{4} \rfloor$.

\pagebreak[2]

(b) Put $l= \lceil \tfrac p 2 \rceil$. 
If $my \in \mathbb{T}_+$ for all 
$m=1,\ldots,l$, then 
$y \in \{1,\ldots,l\}^\triangleleft =\mathbb{T}_l$, and so
$|y| \leq \tfrac 1 {4l}$. Since $p$ is odd, $\tfrac p 2 < l$.
Thus, $|y| \leq \tfrac 1 {4l} < \tfrac{1} {2p}$.  Therefore, by 
(\ref{eq:Jp:estimate}), one obtains that
\begin{align}
\left| \dfrac{c_1}{p} \right| \leq
|y| + \left|\sum\limits_{i=2}^\infty \dfrac{c_i}{p^i} \right|  <
\dfrac 1{2p} + \dfrac 1 {2p} = \dfrac 1 {p}.
\end{align}
Hence, $c_1=0$, as required.

\pagebreak[2]

(c) Put $l= \lceil \tfrac p 6 \rceil$. 
If $my \in \mathbb{T}_+$ for all 
$m=1,\ldots,l$, then 
$y \in \{1,\ldots,l\}^\triangleleft =\mathbb{T}_l$, and so
$|y| \leq \tfrac 1 {4l}$. Since $p$ is odd, $\tfrac p 6 < l$.
Thus, $|y| \leq \tfrac 1 {4l} < \tfrac{3} {2p}$.  Therefore, by 
(\ref{eq:Jp:estimate}), one obtains that
\begin{align}
\left| \dfrac{c_1}{p} \right| \leq
|y| + \left|\sum\limits_{i=2}^\infty \dfrac{c_i}{p^i} \right|  <
\dfrac 3{2p} + \dfrac 1 {2p} = \dfrac 2 {p}.
\end{align}
Hence, $|c_1| \leq 1$, as required.
\end{proof}

\begin{corollary} \label{cor:Jp:c1}
Let \mbox{$p\geq 5$} be a prime such that \mbox{$p \neq 7$}, and
\mbox{$y=\sum\limits_{i=1}^\infty \frac{c_i}{p^i} \in \mathbb{T}$},
where  \mbox{$c_i\in\mathbb{Z}$} and
\mbox{$|c_i| \leq \frac{p-1}{2}$} for all \mbox{$i \in \mathbb{N}$}.
If $my \in \mathbb{T}_+$ for all $m=1,\ldots, \lfloor\tfrac{p}{4} 
\rfloor$, 
then $c_1\in \{-1,0,1\}$.
\end{corollary}

\begin{proof}
If \mbox{$p \hspace{-1pt}= \hspace{-1pt} 5$}, then 
\mbox{$y \hspace{-1pt} \in \hspace{-1pt} \mathbb{T}_+$}, 
and by Theorem~\ref{thm:Jp:c1}(a),
\mbox{$|c_1| \hspace{-1pt}\leq \hspace{-1pt} \lfloor \frac 7 4\rfloor 
\hspace{-1pt}= \hspace{-1pt}1$}.   If 
\mbox{$p \hspace{-1pt} \geq \hspace{-1.5pt} 11$}, then
\mbox{$\lceil \tfrac p 6 \rceil \hspace{-1pt}\leq \hspace{-1pt}
 \lfloor \tfrac p 4 \rfloor$}, and the 
statement follows by Theorem~\ref{thm:Jp:c1}(c).
\end{proof}

\begin{example}
Let $y=\tfrac 2 7 - \tfrac{3}{49} + \mathbb{Z} = 
\tfrac{11}{49}+ \mathbb{Z}$. Since $\tfrac{11}{49} < 
\tfrac{12}{48}=\tfrac{1}{4}$, clearly $y \in \mathbb{T}_+$, and thus
$my \in \mathbb{T}_+$ for all 
$1 \leq m \leq \lfloor \tfrac 7 4 \rfloor =1$. However, $c_1 =2$ and
$c_2 = -3$. This shows that the assumption that $p\neq 7$ cannot be 
omitted in Corollary~\ref{cor:Jp:c1}. Nevertheless, a slightly weaker 
statement does hold for all primes $p \geq 5$, including $p=7$.
\end{example}

\begin{corollary} \label{cor:Jp:p-1c1}
Let \mbox{$p\geq 5$} be a prime, and 
\mbox{$y=\sum\limits_{i=1}^\infty \frac{c_i}{p^i} \in \mathbb{T}$},
where  \mbox{$c_i\in\mathbb{Z}$} and \mbox{$|c_i| \leq \frac{p-1}{2}$}
for all \mbox{$i \in \mathbb{N}$}. If $my \in \mathbb{T}_+$ for all 
$m=1,\ldots,  \lfloor\tfrac{p}{4} \rfloor$ and
$(p-1)y \in \mathbb{T}_+$, then
$c_1\in \{-1,0,1\}$.
\end{corollary}

\begin{proof}
In light Corollary~\ref{cor:Jp:c1},
it remains to be seen that the statement holds for \mbox{$p=7$}. In this 
case, it is given that $y,6y \in \mathbb{T}_+$, which means that
\begin{align}
y \in \{1,6\}^\triangleleft = \mathbb{T}_6 \cup 
(-\tfrac 1 6+\mathbb{T}_6) \cup (\tfrac 1 6+\mathbb{T}_6).
\end{align}
Thus, $|y| \leq \tfrac 5 {24}$. Therefore, by
(\ref{eq:Jp:estimate}), one obtains that
\begin{align}
\left| \dfrac{c_1}{7} \right| \leq
|y| + \left|\sum\limits_{i=2}^\infty \dfrac{c_i}{7^i} \right|  \leq
\dfrac{5}{24} + \dfrac{1}{14} = \dfrac{47}{168} < \dfrac{48}{168} = 
\dfrac{2}{7}.
\end{align}
Hence, $|c_1| \leq 1$, as desired.
\end{proof}

We turn now to investigating the set $L_{\underline a,p}$, its polar, and 
its quasi-convex hull. Recall 
that the Pontryagin dual $\widehat {\mathbb{J}}_p$ of $\mathbb{J}_p$ is 
the Pr\"ufer group $\mathbb{Z}(p^\infty)$.
For \mbox{$k \in \mathbb{N}$}, let 
\mbox{$\zeta_k\colon \mathbb{J}_p \rightarrow \mathbb{T}$} denote the 
continuous character defined by 
\mbox{$\zeta_k(1)=p^{-(k+1)}$}. For \mbox{$m \in \mathbb{N}$},
put \mbox{$J_{\underline{a},p,m} = \{ k \in \mathbb{N} \mid m\zeta_k 
\hspace{-1pt}\in \hspace{-2pt}
L_{\underline a,p}^\triangleright \}$} and
\mbox{$Q_{\underline{a},p,m} = 
\{m\zeta_ k \mid k \hspace{-1pt}\in \hspace{-2pt}
J_{\underline{a},p,m}\}^\triangleleft$}.

\begin{lemma} \label{lemma:Jp:J12p}
For $1 \leq m \leq p-1$,
\begin{align}
J_{\underline{a},p,m} =
\begin{cases}
\mathbb{N} & \text{if }\ \tfrac m p \in \mathbb{T}_+\\
\mathbb{N} \backslash \underline{a}
 & \text{if }\ \tfrac m p \not\in \mathbb{T}_+.
\end{cases}
\end{align}
\end{lemma}

\begin{proof}
For $1 \leq m \leq p-1$ and $i \in \mathbb{Z}$, 
\begin{align} \label{eq:Zp:pmi}
m\cdot p^{i} \in \mathbb{T}_+  \quad & \Longleftrightarrow \quad
i \neq -1 \vee (i=-1 \wedge \tfrac m p \in \mathbb{T}_+)\\
 & \Longleftrightarrow \quad
i \neq -1 \vee \tfrac m p \in \mathbb{T}_+.
\end{align}
For $k,n \in \mathbb{N}$, one has  $m\zeta_k(x_n) = m \cdot p^{a_n-k-1}$.
Thus, by (\ref{eq:Zp:pmi}) applied to $i=a_n-k-1$,
\begin{align}
m\zeta_k(x_n) \in \mathbb{T}_+ \quad & \Longleftrightarrow \quad
a_n-k-1 \neq -1 \vee  \tfrac m p \in \mathbb{T}_+ \\
& \Longleftrightarrow \quad
k \neq a_n \vee \tfrac m p \in \mathbb{T}_+.
\end{align}
Consequently, $m\zeta_k \in L_{\underline a,p}^\triangleright$ if and only 
if $k \neq a_n$ for all $n \in \mathbb{N}$, or $\tfrac m p \in 
\mathbb{T}_+$. 
\end{proof}

\begin{theorem} \label{thm:Jp:Q12p}
If $p>2$, then 
\mbox{$Q_{\underline{a},p,1} \cap \cdots 
\cap Q_{\underline{a},p,p-1} \subseteq \{
\sum\limits_{n=0}^\infty \varepsilon_n y_n \mid
(\forall n \in \mathbb{N}) (\varepsilon_n \in \{-1,0,1\})\}$}.
\end{theorem}

\begin{proof}
Recall that every element $x \in \mathbb{J}_p$ can be written in the form
$x= \sum\limits_{i=0}^\infty c_i \cdot p^i$, 
where $c_i \in \mathbb{Z}$ and $|c_i| \leq \frac{p-1}{2}$. For $x$ 
represented in this form, for every $k \in \mathbb{N}$, one has
\begin{align} \label{eq:Jp:zetak}
\zeta_k(x) = \sum\limits_{i=0}^\infty c_i \cdot p^{i-k-1} \equiv_1
\sum\limits_{i=0}^k c_i \cdot p^{i-k-1} = 
\sum\limits_{i=1}^{k+1} \frac{c_{k-i+1}}{p^{i}}.
\end{align}
Let $x \in Q_{\underline{a},p,1} \cap \cdots\cap Q_{\underline{a},p,p-1}$. 
If $k\neq a_n$ for any $n \in\mathbb{N}$, then by 
Lemma~\ref{lemma:Jp:J12p},
\begin{align}
k\in \mathbb{N}\backslash \underline a = 
J_{\underline{a},p,1} = \cdots=J_{\underline{a},p,p-1}.
\end{align}
Thus, for $y=\zeta_k(x)$,  one has $my \in \mathbb{T}_+$ for
$m=1,\ldots,p-1$. Therefore, by Theorem~\ref{thm:Jp:c1}(b),
the coefficient of $\tfrac 1 p$ in (\ref{eq:Jp:zetak}) is zero. Hence,
$c_{k}=0$ for every $k$ such that $k\neq a_n$ for all 
$n\in\mathbb{N}$. For $p=3$, this already completes the proof, 
and so we may assume that $p\geq 5$.
If $k=a_n$ for some $n \in \mathbb{N}$,
then by Lemma~\ref{lemma:Jp:J12p},
\begin{align}
k\in \mathbb{N} = J_{\underline{a},p,1} = \cdots=
J_{\underline{a},p,{\lfloor \tfrac p 4\rfloor}}=
J_{\underline{a},p,p-1}.
\end{align} 
Consequently, for $y=\zeta_k(x)$, one has $my \in \mathbb{T}_+$ for
$m=1,\ldots,\lfloor\tfrac p 4\rfloor$ and $(p-1)y\in\mathbb{T}_+$. 
Therefore, by Corollary~\ref{cor:Jp:p-1c1}, the coefficient of 
$\tfrac 1 p$ in (\ref{eq:Jp:zetak}) belongs to $\{-1,0,1\}$. 
So, \mbox{$c_{k}=c_{a_n} \in \{-1,0,1\}$}. Hence, $x$ has the form
\begin{align}
x= \sum\limits_{n=0}^\infty \dfrac{c_{a_n}}{p^{a_n}} = 
\sum\limits_{n=0}^\infty \varepsilon_n x_n,
\end{align}
where $\varepsilon_n = c_{a_n} \in \{-1,0,1\}$.
\end{proof}

Put $L_p=L_{\mathbb{N},p}$, that is, the special case of 
$L_{\underline{a},p}$, where $\underline a$ is the sequence 
$a_n=n$. First, we show that if $p \geq 5$, then $L_p$ is quasi-convex.

\begin{lemma} \label{lemma:Jp:polar}
Let $p \geq 5$ be a prime. Then
\mbox{$m_1 \zeta_{k_1} + \cdots + m_l \zeta_{k_l} \in L_p^\triangleright$}
\ for every
\mbox{$0 \leq k_1 < \cdots < k_l$} and \mbox{$m_1,\ldots,m_l$} such that
\mbox{$|m_i| \leq \lfloor \tfrac p 4 \rfloor$} for all 
\mbox{$i \leq i \leq l$}.
\end{lemma}

\pagebreak[2]

\begin{proof}
Let  $p^n \in L_p$. We may assume that \mbox{$n \leq k_1$},
because \mbox{$\eta_{k_i}(p^n)=0$} for every \mbox{$i\in\mathbb{N}$} such 
that \mbox{$k_i<n$}. Since  \mbox{$0\leq k_1 < \cdots < k_l$},
\begin{align}
\frac{1}{p^{k_1+1}} + \cdots + \frac{1}{p^{k_l+1}} \leq
\frac{1}{p^{k_1+1}} + \frac{1}{p^{k_1+2}}+ \cdots + \frac{1}{p^{k_1+l}}
< \sum\limits_{i=0}^\infty \frac{1}{p^{k_1+i+1}} = 
\dfrac{1}{p^{k_1}(p-1)}.
\end{align}
One has $|m_i| \leq \lfloor \tfrac p 4 \rfloor \leq \tfrac{p-1}{4}$ for 
each $i$, because $p$ is odd. Thus,
\begin{align}
|(m_1 \zeta_{k_1} + \cdots + m_l \zeta_{k_l})(p^n)| & \leq
|m_1||\eta_{k_1}(p^n)| + \cdots + |m_l||\eta_{k_l}(p^n)| \\
& \leq \dfrac{p-1}4
\left(\frac{1}{p^{k_1+1}} + \cdots + \frac{1}{p^{k_l+1}} \right) p^n \\
& < \dfrac{p-1}4  \cdot
\dfrac{1}{p^{k_1}(p-1)} \cdot p^n = \dfrac{1}{4 p^{k_1 -n}} \leq 
\dfrac 1 4.
\end{align}
Therefore, $(m_1 \zeta_{k_1} + \cdots + m_l \zeta_{k_l})(p^n)\in 
\mathbb{T}_+$ 
for all $n \in \mathbb{N}$. Hence,
\mbox{$m_1\zeta_{k_1} + \cdots + m_l \zeta_{k_l} \in L_p^\triangleright$}, 
as desired.
\end{proof}

\begin{proposition} \label{prop:Jp:Lp}
Let $p \geq 5$ be a prime. Then $L_p$ is quasi-convex.
\end{proposition}

\begin{proof}
Let \mbox{$x \in Q_\mathbb{T}(L_p)\backslash\{0\}$}. Then
\mbox{$x \in  Q_{\mathbb{N},p,1} \cap \cdots\cap Q_{\mathbb{N},p,p-1}$}.
Consequently, by Theorem~\ref{thm:Jp:Q12p},
\mbox{$x=\sum\limits_{n=0}^\infty \varepsilon_n p^{n}$}, where
\mbox{$\varepsilon_n \in \{-1,0,1\}$} for all 
\mbox{$n \in \mathbb{N}$}. Observe that for $k\in \mathbb{N}$,
\begin{align}
\zeta_k(x) = 
\sum\limits_{n=0}^\infty \dfrac{\varepsilon_n}{p^{-n+k+1}} \equiv_1
\sum\limits_{n=0}^{k} \dfrac{\varepsilon_n}{p^{-n+k+1}}
= \sum\limits_{i=1}^{k+1}  \frac{\varepsilon_{k-i+1}}{p^{i}}.
\end{align}
Let $k_1$ denote the smallest index such that  $\varepsilon_{k_1}\neq 0$, 
and let $k_2 \in \mathbb{N}$ be such that $k_1 < k_2$. In order to prove 
that $x \in K_p$, it remains to be seen that $\varepsilon_{k_2}=0$. 
In order to simplify notations, we set $\varepsilon_n=0$ for all 
$n \in \mathbb{Z}\backslash\mathbb{N}$. Put 
\begin{align} \label{eq:Jp:y+}
y_+ &=(\zeta_{k_1}+\zeta_{k_2})(x) = 
\sum\limits_{i=1}^{k_1+1}  \frac{\varepsilon_{k_1-i+1}}{p^{i}}
+ \sum\limits_{i=1}^{k_2+1}  \frac{\varepsilon_{k_2-i+1}}{p^{i}} =
\sum\limits_{i=1}^{k_2+1}  
\frac{\varepsilon_{k_1-i+1}+\varepsilon_{k_2-i+1}}{p^{i}}, \text{ and}\\
\label{eq:Jp:y-}
y_- &=(\zeta_{k_1} - \zeta_{k_2})(x) = 
\sum\limits_{i=1}^{k_1+1}  \frac{\varepsilon_{k_1-i+1}}{p^{i}}
- \sum\limits_{i=1}^{k_2+1}  \frac{\varepsilon_{k_2-i+1}}{p^{i}} =
\sum\limits_{i=1}^{k_2+1}
\frac{\varepsilon_{k_1-i+1}-\varepsilon_{k_2-i+1}}{p^{i}}.
\end{align}
By Lemma~\ref{lemma:Jp:polar},
$m\zeta_{k_1}+m\zeta_{k_2}, m\zeta_{k_1}-m\zeta_{k_2} 
\in L_p^\triangleright$ 
for $m=1,\ldots,\lfloor \tfrac p 4 \rfloor$. Thus, for 
$m=1,\ldots,\lfloor \tfrac p 4 \rfloor$,
\begin{align}
\label{eq:Jp:my+}
my_+ & = m (\zeta_{k_1}+\zeta_{k_2})(x) \in \mathbb{T}_+, \text{ and}\\
\label{eq:Jp:my-}
my_- & = m (\zeta_{k_1}-\zeta_{k_2})(x) \in \mathbb{T}_+.
\end{align}
One has 
$|\varepsilon_{k_1-i+1} + \varepsilon_{k_2-i+1}| \leq 2 \leq \tfrac{p-1}2$
and 
$|\varepsilon_{k_1-i+1} - \varepsilon_{k_2-i+1}| \leq 2 \leq 
\tfrac{p-1}2$,
for all $i \in \mathbb{Z}$, because  
\mbox{$|\varepsilon_n| \leq 1$} for all 
\mbox{$n \in \mathbb{Z}$}. By replacing $x$ with $-x$ if necessary, we may 
assume that \mbox{$\varepsilon_{k_1}\hspace{-2pt}=1$}. Put 
\mbox{$\rho=\varepsilon_{k_2}$}. We 
may apply 
Corollary~\ref{cor:Jp:c1} and~\ref{cor:Jp:p-1c1}, but we must distinguish 
between three (overlapping) cases:

1. If $p\neq 7$, then by Corollary~\ref{cor:Jp:c1}, the coefficients
of $\tfrac 1 p$ in (\ref{eq:Jp:y+}) and (\ref{eq:Jp:y-}) are $-1$, $0$, or 
$1$. In other words, $|1 + \rho| \leq 1$ and $|1 - \rho| \leq 1$. Since
$\rho \in \{-1,0,1\}$, this implies that $\rho=0$, as required.

\pagebreak[2]

2. If $k_1+1 < k_2$, then by Lemma~\ref{lemma:Jp:polar},
\begin{align}
(p-1)(\eta_{k_1}+\eta_{k_2}) & = 
\zeta_{k_1-1}-\zeta_{k_1}  +\zeta_{k_2-1}-\zeta_{k_2} \in L_p^\triangleright,
\text{ and} \\
(p-1)(\eta_{k_1}-\eta_{k_2}) & =
\zeta_{k_1-1}-\zeta_{k_1} -\zeta_{k_2-1}+\zeta_{k_2} \in L_p^\triangleright.
\end{align}
(If $k_1=0$, then of course, $\zeta_{k_1-1} = \zeta_{-1} =0$.)
Thus, in addition to (\ref{eq:Jp:my+}) and (\ref{eq:Jp:my-}), one also has
\begin{align}
(p-1)y_+ & = (p-1) (\zeta_{k_1}+\zeta_{k_2})(x) \in \mathbb{T}_+, \text{ and}\\
(p-1)y_- & = (p-1) (\zeta_{k_1}-\zeta_{k_2})(x) \in \mathbb{T}_+.
\end{align}
Therefore, by Corollary~\ref{cor:Jp:p-1c1}, the coefficients of $\tfrac 1 
p$ 
in (\ref{eq:Jp:y+}) and (\ref{eq:Jp:y-}) are $-1$, $0$, or $1$. 
In other words, $|1 + \rho| \leq 1$ and $|1 - \rho| \leq 1$. Since
$\rho \in \{-1,0,1\}$, this implies that $\rho=0$, as required.

3. If $p=7$ and $k_2=k_1+1$, then by what we have shown so far, 
\begin{align}
x = 7^{k_1} +\rho\cdot 7^{k_1+1} = (1+7\rho)\cdot 7^{k_1},
\end{align}
where $\rho \in \{-1,0,1\}$. By Lemma~\ref{lemma:Jp:polar}, 
\begin{align}
(7\rho +1)\zeta_{k_1+1} = \rho\zeta_{k_1} +  \zeta_{k_1+1} 
\in L_7^\triangleright.
\end{align}
Therefore,
\begin{align}
\dfrac{2\rho}{7} + \dfrac{1}{49} =
\dfrac{14\rho +1 }{49} \equiv_1
\dfrac{(7\rho +1)^2}{49} =
(7\rho +1)\zeta_{k_1+1}(x)  \in \mathbb{T}_+.
\end{align}
Hence, $\rho=0$, as required. 
\end{proof}

\begin{proof}[Proof of Theorem~\ref{thm:Jp:main}.]
By Proposition~\ref{prop:Jp:Lp}, $L_p$ is quasi-convex.
Thus,  $L_{\underline{a},p} \subseteq L_p$ implies that
$Q_\mathbb{T}(L_{\underline{a},p}) \subseteq L_p$. On the other hand,
by Theorem~\ref{thm:Jp:Q12p},
\begin{align}
Q_\mathbb{T}(L_{\underline{a},p}) \subseteq
Q_{\underline{a},p,1} \cap \cdots\cap Q_{\underline{a},p,p-1}\subseteq  \{
\sum\limits_{n=0}^\infty \varepsilon_n y_n \mid
(\forall n \in \mathbb{N}) (\varepsilon_n \in \{-1,0,1\})\}.
\end{align}
Therefore,
\begin{align}
Q_\mathbb{T}(L_{\underline{a},p}) \subseteq
L_p \cap \{\sum\limits_{n=0}^\infty \varepsilon_n y_n \mid
(\forall n \in \mathbb{N}) (\varepsilon_n \in \{-1,0,1\})\} = 
L_{\underline{a},p},
\end{align}
as desired.
\end{proof}

\begin{example}\label{ex:Jp:p23}
If $p=2$ or $p=3$, then $Q_{\mathbb{J}_p}(L_p)=\mathbb{J}_p$, that is,
$L_p$ is $qc$-dense in $\mathbb{J}_p$ (cf.~\cite[4.6(c)]{DikLeo}). In 
particular, $L_2$ and $L_3$ are not quasi-convex in $\mathbb{J}_2$ and 
$\mathbb{J}_3$, respectively. Thus, 
Proposition~\ref{prop:Jp:Lp} fails if $p \not\geq 5$. Therefore, 
Theorem~\ref{thm:Jp:main} does not hold for $p \not\geq 5$. 
\end{example}

\section{Sequences of the form $\boldsymbol{\{p^{-(a_n+1)}\}_{n=1}^\infty}$ in
$\boldsymbol{\mathbb{T}}$}

\label{sect:Tp}

In this section, we prove Theorem~\ref{thm:Tp:main}. Recall  that 
the Pontryagin dual $\widehat{\mathbb{T}}$ of 
$\mathbb{T}$ is $\mathbb{Z}$. For \mbox{$k \in \mathbb{N}$}, let 
\mbox{$\eta_k\colon \mathbb{T} \rightarrow \mathbb{T}$} denote the 
continuous character defined by 
\mbox{$\eta_k(x)=p^{k}\cdot x$}. For \mbox{$m \in \mathbb{N}$},
put 
\begin{align}
J_{\underline{a},p,m} = \{ k \in \mathbb{N} \mid m\eta_k \in 
K_{\underline a,p}^\triangleright \} \quad \text{ and } \quad
Q_{\underline{a},p,m} = \{m\eta_ k \mid k \in 
J_{\underline{a},p,m}\}^\triangleleft.
\end{align}

\begin{lemma} \label{lemma:Tp:J12p}
For $1 \leq m \leq p-1$, 
\begin{align}
J_{\underline{a},p,m} = 
\begin{cases}
\mathbb{N} & \text{if }\ \tfrac m p \in \mathbb{T}_+\\
\mathbb{N} \backslash \underline{a}  
 & \text{if }\ \tfrac m p \not\in \mathbb{T}_+.
\end{cases}
\end{align}
\end{lemma}

\begin{proof}
For $1 \leq m \leq p-1$ and $i \in \mathbb{Z}$, 
\begin{align} \label{eq:Tp:pmi}
m\cdot p^{i} \in \mathbb{T}_+  \quad & \Longleftrightarrow \quad
i \neq -1 \vee (i=-1 \wedge \tfrac m p \in \mathbb{T}_+)\\
 & \Longleftrightarrow \quad
i \neq -1 \vee \tfrac m p \in \mathbb{T}_+.
\end{align}
For $k,n \in \mathbb{N}$, one has  $m\eta_k(x_n) = m \cdot p^{k-a_n-1}$.
Thus, by (\ref{eq:Tp:pmi}) applied to $i=k-a_n-1$,
\begin{align}
m\eta_k(x_n) \in \mathbb{T}_+ \quad & \Longleftrightarrow \quad
k-a_n-1 \neq -1 \vee  \tfrac m p \in \mathbb{T}_+ \\
& \Longleftrightarrow \quad
k \neq a_n \vee \tfrac m p \in \mathbb{T}_+.
\end{align}
Consequently, $m\eta_k \in K_{\underline a,p}^\triangleright$ if and only 
if $k \neq a_n$ for all $n \in \mathbb{N}$, or $\tfrac m p \in 
\mathbb{T}_+$. 
\end{proof}

\begin{theorem} \label{thm:Tp:Q12p}
If $p >2$, then 
\mbox{$Q_{\underline{a},p,1} \cap \cdots 
\cap Q_{\underline{a},p,p-1} \subseteq \{
\sum\limits_{n=0}^\infty \varepsilon_n x_n \mid
(\forall n \in \mathbb{N}) (\varepsilon_n \in \{-1,0,1\})\}$}.
\end{theorem}

\begin{proof}
Let  
\mbox{$x=\sum\limits_{i=1}^\infty \dfrac{c_i}{p^i}$} be a representation of 
\mbox{$x\in \mathbb{T}$}. 
Observe that for every \mbox{$k\in \mathbb{N}$}, one has
\begin{align} \label{eq:Tp:etak}
\eta_{k}(x) = p^k  x = 
\sum\limits_{i=1}^\infty \dfrac{c_i}{p^{i-k}} \equiv_1 
\sum\limits_{i=k+1}^\infty \dfrac{c_i}{p^{i-k}} = 
\sum\limits_{i=1}^\infty \dfrac{c_{k+i}}{p^i}.
\end{align}
Let $x \in Q_{\underline{a},p,1} \cap \cdots\cap Q_{\underline{a},p,p-1}$. If
 $k\neq a_n$ for any $n \in\mathbb{N}$, then by Lemma~\ref{lemma:Tp:J12p},
\begin{align}
k\in \mathbb{N}\backslash \underline a = J_{\underline{a},p,1} = 
\cdots=J_{\underline{a},p,p-1}.
\end{align}
Thus, for $y=\eta_k(x)$,  one has $my \in \mathbb{T}_+$ for
$m=1,\ldots,p-1$. 
Arguing as in the proof of Theorem~\ref{thm:Jp:Q12p}, we conclude that
\mbox{$c_{k+1}=0$} for every $k$ such that $k\neq a_n$ for all 
\mbox{$n\in\mathbb{N}$}. For $p=3$, this already completes the proof,
and so we may assume that $p\geq 5$.
If $k=a_n$ for some $n \in \mathbb{N}$,
then by Lemma~\ref{lemma:Tp:J12p},
\begin{align}
k\in \mathbb{N} = J_{\underline{a},p,1} = \cdots=
J_{\underline{a},p,{\lfloor \tfrac p 4\rfloor}}=
J_{\underline{a},p,p-1}.
\end{align}  
Arguing further as in the proof of Theorem~\ref{thm:Jp:Q12p} (but making 
use of the characters $\eta_k$ instead of $\zeta_k$), we conclude that
\mbox{$c_{k+1}=c_{a_n+1} \in \{-1,0,1\}$}.
Hence, $x$ has the form
\begin{align}
x= \sum\limits_{n=0}^\infty \dfrac{c_{a_n+1}}{p^{a_n+1}} = 
\sum\limits_{n=0}^\infty \varepsilon_n x_n,
\end{align}
where $\varepsilon_n = c_{a_n+1} \in \{-1,0,1\}$.
\end{proof}

Put $K_p =K_{\mathbb{N},p}$, that is, the special case of 
$K_{\underline{a},p}$, where $\underline a$ is the sequence 
$a_n=n$. First, we show that if $p \geq 5$, then $K_p$ is quasi-convex.

\begin{lemma} \label{lemma:Tp:polar}
Let $p \geq 5$ be a prime. Then
\mbox{$m_1 \eta_{k_1} + \cdots + m_l \eta_{k_l} \in K_p^\triangleright$}
\ for every
\mbox{$0\leq k_1 < \cdots < k_l$} and \mbox{$m_1,\ldots,m_l$} such that
\mbox{$|m_i| \leq \lfloor \tfrac p 4 \rfloor$} for all 
\mbox{$i \leq i \leq l$}.
\end{lemma}

\begin{proof}
Let $\tfrac 1 {p^{n+1}} \in K_p$. We may assume that \mbox{$k_l \leq n$}, 
because \mbox{$\eta_{k_i}(\tfrac 1 {p^{n+1}})=0$} for every 
\mbox{$i\in\mathbb{N}$} such that \mbox{$n<k_i$}. Since 
\mbox{$0\leq k_1 < \cdots < k_l$},
\begin{align}
p^{k_1} + \cdots +p^{k_l} \leq 1 + p + \cdots + p^{k_l} = 
\dfrac{p^{k_l+1}-1}{p-1}.
\end{align}
One has $|m_i| \leq \lfloor \tfrac p 4 \rfloor \leq \tfrac{p-1}{4}$ for 
each $i$, because $p$ is odd. Thus,
\begin{align}
|(m_1 \eta_{k_1} + \cdots + m_l \eta_{k_l})(\tfrac{1}{p^{n+1}})| & \leq
|m_1||\eta_{k_1}(\tfrac{1}{p^{n+1}})| + \cdots + 
|m_l||\eta_{k_l}(\tfrac{1}{p^{n+1}})| \\
& \leq \dfrac{p-1}4  
(p^{k_1} + \cdots +p^{k_l}) \dfrac{1}{p^{n+1}} \\
& \leq \dfrac{p-1}4  \cdot \dfrac{p^{k_l+1}-1}{p-1} \cdot
\dfrac{1}{p^{n+1}} = 
\dfrac{p^{k_l+1}-1}{4p^{n+1}} < \dfrac 1 4.
\end{align}
Therefore, $(m_1 \eta_{k_1} + \cdots + m_l \eta_{k_l})(\tfrac{1}{p^{n+1}})
\in \mathbb{T}_+$ for all $n \in \mathbb{N}$. Hence,
\mbox{$m_1 \eta_{k_1} + \cdots + m_l \eta_{k_l} \in K_p^\triangleright$}, 
as desired.
\end{proof}

\begin{proposition} \label{prop:Tp:Kp}
Let $p \geq 5$ be a prime. Then $K_p$ is quasi-convex.
\end{proposition}

\begin{proof}
Let \mbox{$x \in Q_\mathbb{T}(K_p)\backslash\{0\}$}. Then
\mbox{$x \in  Q_{\mathbb{N},p,1} \cap \cdots\cap Q_{\mathbb{N},p,p-1}$}.
Consequently, by Theorem~\ref{thm:Tp:Q12p},
\mbox{$x=\sum\limits_{n=0}^\infty \tfrac{\varepsilon_n}{p^{n+1}}$}, where
\mbox{$\varepsilon_n \in \{-1,0,1\}$} for all 
\mbox{$n \in \mathbb{N}$}. Observe that for $k\in \mathbb{N}$,
\begin{align}
\eta_k(x) = p^k x = 
\sum\limits_{n=0}^\infty \dfrac{\varepsilon_n}{p^{n-k+1}} \equiv_1
\sum\limits_{n=k}^\infty \dfrac{\varepsilon_n}{p^{n-k+1}}
= \sum\limits_{i=1}^\infty  \frac{\varepsilon_{k+i-1}}{p^{i}}.
\end{align}
Let $k_1$ denote the smallest index such that  $\varepsilon_{k_1}\neq 0$, 
and let $k_2 \in \mathbb{N}$ be such that $k_1 < k_2$. In order to prove 
that $x \in K_p$, it remains to be seen that $\varepsilon_{k_2}=0$. 
Put 
\begin{align} \label{eq:Tp:y+}
y_+ &=(\eta_{k_1}+\eta_{k_2})(x) = 
\sum\limits_{i=1}^\infty  \frac{\varepsilon_{k_1+i-1}}{p^{i}}+
\sum\limits_{i=1}^\infty  \frac{\varepsilon_{k_2+i-1}}{p^{i}} = 
\sum\limits_{i=1}^\infty  
\frac{\varepsilon_{k_1+i-1}+\varepsilon_{k_2+i-1}}{p^{i}}, \text{ and}\\
\label{eq:Tp:y-}
y_- &=(\eta_{k_1} - \eta_{k_2})(x) = 
\sum\limits_{i=1}^\infty  \frac{\varepsilon_{k_1+i-1}}{p^{i}} -
\sum\limits_{i=1}^\infty  \frac{\varepsilon_{k_2+i-1}}{p^{i}} = 
\sum\limits_{i=1}^\infty  
\frac{\varepsilon_{k_1+i-1} - \varepsilon_{k_2+i-1}}{p^{i}}.
\end{align}
By Lemma~\ref{lemma:Tp:polar},
$m\eta_{k_1}+m\eta_{k_2}, m\eta_{k_1}-m\eta_{k_2} 
\in K_p^\triangleright$ 
for $m=1,\ldots,\lfloor \tfrac p 4 \rfloor$. Thus, for 
$m=1,\ldots,\lfloor \tfrac p 4 \rfloor$,
\begin{align}
\label{eq:Tp:my+}
my_+ & = m (\eta_{k_1}+\eta_{k_2})(x) \in \mathbb{T}_+, \text{ and}\\
\label{eq:Tp:my-}
my_- & = m (\eta_{k_1}-\eta_{k_2})(x) \in \mathbb{T}_+.
\end{align}
One has 
$|\varepsilon_{k_1+i-1} + \varepsilon_{k_2+i-1}| \leq 2 \leq \tfrac{p-1}2$
and 
$|\varepsilon_{k_1+i-1} - \varepsilon_{k_2+i-1}| \leq 2 \leq \tfrac{p-1}2$,
for all $i \in \mathbb{N}$, because  
\mbox{$|\varepsilon_n| \leq 1$} for all 
\mbox{$n \in \mathbb{N}$}.  By replacing $x$ with $-x$ if necessary, we 
may assume that 
\mbox{$\varepsilon_{k_1}\hspace{-2pt}= 1$}. Put 
\mbox{$\rho=\varepsilon_{k_2}$}.
We may apply Corollaries~\ref{cor:Jp:c1} and~\ref{cor:Jp:p-1c1}, but we 
must distinguish between three (overlapping) cases:

1. If $p\neq 7$, then by Corollary~\ref{cor:Jp:c1}, the coefficients
of $\tfrac 1 p$ in (\ref{eq:Tp:y+}) and (\ref{eq:Tp:y-}) are $-1$, $0$, or 
$1$. In other words, $|1 + \rho| \leq 1$ and $|1 - \rho| \leq 1$. Since
$\rho \in \{-1,0,1\}$, this implies that $\rho=0$, as required.

2. If $k_1+1 < k_2$, then by Lemma~\ref{lemma:Tp:polar},
\begin{align}
(p-1)(\eta_{k_1}+\eta_{k_2}) & = 
-\eta_{k_1}+\eta_{k_1+1}  -\eta_{k_2}+\eta_{k_2+1} \in K_p^\triangleright,
\text{ and} \\
(p-1)(\eta_{k_1}-\eta_{k_2}) & =
-\eta_{k_1}+\eta_{k_1+1} +\eta_{k_2}-\eta_{k_2+1} \in K_p^\triangleright.
\end{align}
Thus, in addition to (\ref{eq:Tp:my+}) and (\ref{eq:Tp:my-}), one also 
has
\begin{align}
(p-1)y_+ & = (p-1) (\eta_{k_1}+\eta_{k_2})(x) \in \mathbb{T}_+, \text{ and}\\
(p-1)y_- & = (p-1) (\eta_{k_1}-\eta_{k_2})(x) \in \mathbb{T}_+.
\end{align}
Therefore, by Corollary~\ref{cor:Jp:p-1c1}, the coefficients of 
$\tfrac 1 p$ in (\ref{eq:Tp:y+}) and (\ref{eq:Tp:y-}) are $-1$, $0$, or $1$. 
In other words, $|1 + \rho| \leq 1$ and $|1 - \rho| \leq 1$. Since
$\rho \in \{-1,0,1\}$, this implies that $\rho=0$, as required.

3. If $p=7$ and $k_2=k_1+1$, then by what we have shown so far, 
\begin{align}
x = \frac{1}{7^{k_1+1}} +\frac{\rho}{7^{k_1+2}} =
\dfrac{7+\rho}{7^{k_1+2}},
\end{align}
where $\rho \in \{-1,0,1\}$. By 
Lemma~\ref{lemma:Tp:polar}, 
\begin{align}
(7+\rho)\eta_{k_1} = \rho \eta_{k_1} + \eta_{k_1+1} \in 
K_7^\triangleright.
\end{align}
Therefore, 
\begin{align}
\dfrac{2\rho}{7} + \dfrac{\rho^2}{49} =
\dfrac{14\rho +\rho^2}{49} \equiv_1
\dfrac{(7+\rho)^2}{49} =
(7+\rho)\eta_{k_1}(x)  \in \mathbb{T}_+.
\end{align}
Hence, $\rho=0$, as required. 
\end{proof}

\begin{proof}[Proof of Theorem~\ref{thm:Tp:main}.]
By Proposition~\ref{prop:Tp:Kp}, $K_p$ is quasi-convex.
Thus,  $K_{\underline{a},p} \subseteq K_p$ implies that
$Q_\mathbb{T}(K_{\underline{a},p}) \subseteq K_p$. On the other hand,
by Theorem~\ref{thm:Tp:Q12p},
\begin{align}
Q_\mathbb{T}(K_{\underline{a},p}) \subseteq
Q_{\underline{a},p,1} \cap \cdots\cap Q_{\underline{a},p,p-1}\subseteq  \{
\sum\limits_{n=0}^\infty \varepsilon_n x_n \mid
(\forall n \in \mathbb{N}) (\varepsilon_n \in \{-1,0,1\})\}.
\end{align}
Therefore,
\begin{align}
Q_\mathbb{T}(K_{\underline{a},p}) \subseteq
K_p \cap \{
\sum\limits_{n=0}^\infty \varepsilon_n x_n \mid
(\forall n \in \mathbb{N}) (\varepsilon_n \in \{-1,0,1\})\} = 
K_{\underline{a},p},
\end{align}
as desired.
\end{proof}

\begin{example}\label{ex:Tp:p23}
If $p=2$ or $p=3$, then $Q_{\mathbb{T}}(K_p)=\mathbb{T}$, that is,
$K_p$ is $qc$-dense in $\mathbb{T}$ (cf.~\cite[4.4]{DikLeo}). In
particular, $K_2$ and $K_3$ are not quasi-convex in $\mathbb{T}$. Thus,
Proposition~\ref{prop:Tp:Kp} fails if $p \not\geq 5$. Therefore,
Theorem~\ref{thm:Tp:main} does not hold for $p \not\geq 5$.
\end{example}

\section*{Acknowledgments}

We are grateful to Karen Kipper for her kind help in proof-reading this
paper for grammar and punctuation.

\nocite{Aussenhofer}   
\nocite{Banasz}        
\nocite{BeiLeoDikSte}  
\nocite{deLeoPhD}      
\nocite{DikLeo}        
\nocite{Fuchs}
\nocite{Pontr}         
\nocite{Vilenkin}      
\nocite{GLdualtheo}

{\footnotesize

\bibliography{notes,notes2,notes3}

\def\cprime{$'$} \def\cprime{$'$}
\begin{thebibliography}{10}

\bibitem{ArensDLin}
R.~Arens.
\newblock Duality in linear spaces.
\newblock {\em Duke Math. J.}, 14:787--794, 1947.

\bibitem{ArmacostLCA}
D.~L. Armacost.
\newblock {\em The structure of locally compact abelian groups}, volume~68 of
  {\em Monographs and Textbooks in Pure and Applied Mathematics}.
\newblock Marcel Dekker Inc., New York, 1981.

\bibitem{Aussenhofer}
L.~Aussenhofer.
\newblock Contributions to the duality theory of abelian topological groups and
  to the theory of nuclear groups.
\newblock {\em Dissertationes Math. (Rozprawy Mat.)}, 384:113, 1999.

\bibitem{Banasz}
W.~Banaszczyk.
\newblock {\em Additive subgroups of topological vector spaces}, volume 1466 of
  {\em Lecture Notes in Mathematics}.
\newblock Springer-Verlag, Berlin, 1991.

\bibitem{BeiLeoDikSte}
M.~Beiglb{\"o}ck, L.~de~Leo, D.~Dikranjan, and C.~Steineder.
\newblock On quasi-convexity of finite sets.
\newblock {\em Preprint}, March 2005.

\bibitem{BrugMar2}
M.~Bruguera and E.~Mart{\'{\i}}n-Peinador.
\newblock Banach-{D}ieudonn\'e theorem revisited.
\newblock {\em J. Aust. Math. Soc.}, 75(1):69--83, 2003.

\bibitem{ChaMarTar}
M.~J. Chasco, E.~Mart{\'{\i}}n-Peinador, and V.~Tarieladze.
\newblock On {M}ackey topology for groups.
\newblock {\em Studia Math.}, 132(3):257--284, 1999.

\bibitem{DikLeo}
D.~Dikranjan and L.~de~Leo.
\newblock Countably infinite quasi-convex sets in some locally compact abelian
  groups.
\newblock {\em Topology Appl.}, to appear.

\bibitem{DikGL1}
D.~Dikranjan and G.~Luk\'acs.
\newblock Quasi-convex sequences in the circle and the $3$-adic integers.
\newblock {\em Submitted}, 2008.
\newblock ArXiv: 0811.1966.

\bibitem{DikProET}
D.~Dikranjan and I.~Prodanov.
\newblock A class of compact abelian groups.
\newblock {\em Annuaire Univ. Sofia Fac. Math. M\'ec.}, 70:191--206 (1981),
  1975/76.

\bibitem{DikProSto}
D.~Dikranjan, I.~R. Prodanov, and L.~N. Stoyanov.
\newblock {\em Topological groups}, volume 130 of {\em Monographs and Textbooks
  in Pure and Applied Mathematics}.
\newblock Marcel Dekker Inc., New York, 1990.
\newblock Characters, dualities and minimal group topologies.

\bibitem{Fuchs}
L.~Fuchs.
\newblock {\em Infinite abelian groups. {V}ol. {I}}.
\newblock Pure and Applied Mathematics, Vol. 36. Academic Press, New York,
  1970.

\bibitem{Hern2}
S.~Hern{\'a}ndez.
\newblock Pontryagin duality for topological abelian groups.
\newblock {\em Math. Z.}, 238(3):493--503, 2001.

\bibitem{HofMor}
K.~H. Hofmann and S.~A. Morris.
\newblock {\em The structure of compact groups}, volume~25 of {\em de Gruyter
  Studies in Mathematics}.
\newblock Walter de Gruyter \& Co., Berlin, 1998.
\newblock A primer for the student---a handbook for the expert.

\bibitem{deLeoPhD}
L.~de~Leo.
\newblock {\em Weak and strong topologies in topological abelian groups}.
\newblock PhD thesis, Universidad Complutense de Mardid, April 2008.

\bibitem{GLdualtheo}
G.~Luk{\'a}cs.
\newblock Notes on duality theories of abelian groups ({C}hapter {I}).
\newblock ArXiv: math.GN/0605149, Dalhousie University 2006.

\bibitem{MackeyCTLS}
G.~W. Mackey.
\newblock On convex topological linear spaces.
\newblock {\em Trans. Amer. Math. Soc.}, 60:519--537, 1946.

\bibitem{Pontr}
L.~S. Pontryagin.
\newblock {\em Selected works. {V}ol. 2}.
\newblock Gordon \& Breach Science Publishers, New York, third edition, 1986.
\newblock Topological groups, Edited and with a preface by R. V. Gamkrelidze,
  Translated from the Russian and with a preface by Arlen Brown, With
  additional material translated by P. S. V. Naidu.

\bibitem{Vilenkin}
N.~Y. Vilenkin.
\newblock The theory of characters of topological {A}belian groups with
  boundedness given.
\newblock {\em Izvestiya Akad. Nauk SSSR. Ser. Mat.}, 15:439--462, 1951.

\end{thebibliography}
}

\begin{samepage}

\bigskip
\noindent
\begin{tabular}{l @{\hspace{1.8cm}} l}
Department of Mathematics and Computer Science & Department of Mathematics\\
University of Udine & University of Manitoba\\
Via delle Scienze, 208 -- Loc. Rizzi, 33100 Udine
 & Winnipeg, Manitoba, R3T 2N2 \\
Italy & Canada \\ & \\
\em e-mail: dikranja@dimi.uniud.it  &
\em e-mail: lukacs@cc.umanitoba.ca
\end{tabular}

\end{samepage}

\end{document}